\documentclass[a4paper,10pt]{article}
\usepackage{amssymb,amsmath,amsthm}
\usepackage{fullpage}
\usepackage{hyperref}
\usepackage{eucal}
\usepackage{color}
\usepackage{stackrel}
\usepackage{graphicx}
\newcommand{\ie}{\emph{i.e.}}
\newcommand{\eg}{\emph{e.g.}}
\newcommand{\cf}{\emph{cf.}}

\newcommand{\Real}{\mathbb{R}}
\newcommand{\Com}{\mathbb{C}}
\newcommand{\Nat}{\mathbb{N}}

\newcommand{\Dom}{\mathsf{D}}
\newcommand{\Ker}{\mathsf{N}}

\newcommand{\eps}{\varepsilon}
\newcommand{\sii}{L^2}

\newcommand{\der}{\mathrm{d}}

\newcommand{\ademi}{\mbox{$\frac{a}{2}$}}
\newcommand{\bdemi}{\mbox{$\frac{b}{2}$}}
\newcommand{\demi}{\mbox{$\frac{1}{2}$}}

\newenvironment{psmallmatrix}
  {\left(\begin{smallmatrix}}
  {\end{smallmatrix}\right)}
\begingroup
    \makeatletter
    \@for\theoremstyle:=definition,remark,plain\do{%
        \expandafter\g@addto@macro\csname th@\theoremstyle\endcsname{%
            \addtolength\thm@preskip\parskip
            }%
        }
\endgroup
\newtheorem{Theorem}{Theorem}
\newtheorem{Lemma}{Lemma}

\newtheorem{Corollary}{Corollary}
\newtheorem{Conjecture}{Conjecture}
\theoremstyle{definition}
\newtheorem{Remark}{Remark}

\def\OMIT#1{}
\usepackage[normalem]{ulem}
\definecolor{DarkGreen}{rgb}{0,0.5,0.1} % David

\newcommand\soutD{\bgroup\markoverwith
{\textcolor{DarkGreen}{\rule[.5ex]{2pt}{1pt}}}\ULon}
\newcommand\soutP{\bgroup\markoverwith
{\textcolor{blue}{\rule[.5ex]{2pt}{1pt}}}\ULon}
\newcommand{\Hm}[1]{\leavevmode{\marginpar{\tiny%
$\hbox to 0mm{\hspace*{-0.5mm}$\leftarrow$\hss}%
\vcenter{\vrule depth 0.1mm height 0.1mm width \the\marginparwidth}%
\hbox to
0mm{\hss$\rightarrow$\hspace*{-0.5mm}}$\\\relax\raggedright #1}}}

\begin{document}
%
%-------%
% TITLE %
%-------%
%------------------------------------------%
%------------------------------------------%
\title{\textbf{\LARGE
Spectral optimisation of Dirac rectangles
}}
\author{Philippe Briet\,$^a$ \ and \ David Krej\v{c}i\v{r}{\'\i}k\,$^b$}
\date{\small 
\vspace{-5ex}
\begin{quote}
\emph{
\begin{itemize}
\item[$a)$] 
Aix-Marseille Université, 
Université de Toulon, CNRS, CPT, Marseille, France;
briet@univ-tln.fr.%
\item[$b)$] 
Department of Mathematics, Faculty of Nuclear Sciences and 
Physical Engineering, Czech Technical University in Prague, 
Trojanova 13, 12000 Prague 2, Czechia;
david.krejcirik@fjfi.cvut.cz.%
\end{itemize}
}
\end{quote}
15 December 2021
}
\maketitle
\vspace{-5ex} 
\begin{abstract}
\noindent
We are concerned with the dependence of the lowest
positive eigenvalue of the Dirac operator on the geometry of rectangles,
subject to infinite-mass boundary conditions. 
We conjecture that the square is a global minimiser 
both under the area or perimeter constraints.
Contrary to well-known non-relativistic analogues,
we show that
the present spectral problem does not admit explicit solutions.
We prove partial optimisation results based on a variational reformulation
and newly established lower and upper bounds to the Dirac eigenvalue.
We also propose an alternative approach 
based on symmetries of rectangles
and a non-convex minimisation problem;
this implies a sufficient condition 
formulated in terms of a symmetry of the minimiser
which guarantees the conjectured results.
 
%
%\bigskip
%\begin{itemize}
%\item[\textbf{Keywords:}]
%\item[\textbf{MSC (2010):}]
%\end{itemize}
%
\end{abstract}
%
%------------------------------------------%
%------------------------------------------%

%---------------------%
\section{Introduction}
%---------------------%
%
Among all membranes of a given area and fixed edges,
the circular one produces the lowest fundamental tone.
This is a well-known interpretation of the celebrated 
Faber--Krahn inequality 
(see, \eg, \cite[Sec.~3]{Henrot})
stating that 
\begin{equation}\label{FK}
  \Lambda_1(\Omega) \geq \Lambda_1(\Omega^*)
  \,,
\end{equation}
where~$\Omega$ is any bounded open planar set, 
$\Omega^*$ is the disk of the same area
and~$\Lambda_1$ is the lowest eigenvalue of 
the boundary value problem 
\begin{equation}\label{Dirichlet} 
\left\{
\begin{aligned}
  -\Delta u &= \Lambda u
  && \mbox{in} && \Omega 
  \,,
  \\
  u &= 0 
  && \mbox{on} && \partial\Omega 
  \,.
\end{aligned}
\right.
\end{equation}
Interpreting the Laplacian as the free Schr\"odinger operator
and the Dirichlet condition as hard-wall boundaries,
the result~\eqref{FK} also says that the ground-state energy
of a non-relativistic quantum particle constrained 
to nanostructures of a given material 
is minimised by the disk. 
By scaling, it is easy to see that~\eqref{FK} 
alternatively holds
under the perimeter constraint 
instead of fixing the area.

Restricting ourselves to rectangles,
it is also true that the square is the optimal geometry
both under the area or perimeter constraints.
More specifically, defining 
\begin{equation}\label{rectangle} 
  \Omega_{a,b} := 
  \left(-\frac{a}{2},\frac{a}{2}\right) 
  \times \left(-\frac{b}{2},\frac{b}{2}\right) 
  ,
\end{equation}
where $a,b$ are any positive numbers,
the inequality~\eqref{FK} remains true 
for $\Omega := \Omega_{a,b}$ and $\Omega^* := \Omega_{a,a}$
whenever $ab=1$ (area constraint) 
or $2(a+b)=4a$ (perimeter constraint).
While the general proof based on symmetrisation techniques applies
to arbitrary quadrilaterals,
the case of rectangles can be alternatively established
in an elementary way just by using the well-known fact that 
the problem~\eqref{Dirichlet} is explicitly solvable
by separation of variables in terms of sine and cosine functions.  
We refer to~\cite{Laugesen_2019} 
for a recent spectral optimisation of the Laplacian eigenvalues  
in the larger generality of rectangular boxes
with Robin boundary conditions.

The purpose of this paper is to investigate whether 
the optimality of the square among all rectangles 
remains true in a relativistic setting,
again both under the area or perimeter constraints.
As simple as it may seem 
(and perhaps heuristically expected), 
the result is far from being obvious.
Indeed, the present relativistic spectral problem 
is not explicitly solvable
and no symmetrisation techniques are available.	 
For these reasons, we are forced to develop alternative approaches.

To state our main results,
consider a relativistic particle of mass $m \geq 0$
constrained to an open connected set $\Omega \subset \Real^2$ 
with locally Lipschitz boundary.
The regularity ensures that the outward unit normal
$
  n = 
  \begin{psmallmatrix}
  n_1 \\ n_2 
  \end{psmallmatrix}
  :\partial\Omega\to\Real^2
$ 
exists almost everywhere.
The quantum Hamiltonian~$H$ acts as the free Dirac operator
\begin{equation}\label{operator1}
  H := 
  \begin{pmatrix}
    m & -i (\partial_1-i\partial_2) \\
    -i(\partial_1+i\partial_2) & -m
  \end{pmatrix}
  \qquad \mbox{in} \qquad
  \sii(\Omega;\Com^2)
  \,.  
\end{equation}
The relativistic analogue of the hard-wall boundaries
are the so-called \emph{infinite-mass boundary conditions}, 
which have attracted a lot of attention recently
\cite{Arrizibalaga-LeTreust-Raymond_2017,
Benguria-Fournais-Stockmeyer-Bosch_2017b,
LeTreust-Ourmieres-Bonafos_2018,
Arrizibalaga-LeTreust-Raymond_2018,
Barbaroux-Cornean-LeTreust-Stockmeyer_2019,
Arrizibalaga-LeTreust-Mas-Raymond_2019}.
We rigorously implement them through the operator domain
\begin{equation}\label{operator2}
  \Dom(H) :=  
  \left\{
  u = 
  \begin{psmallmatrix}
  u_1 \\ u_2 
  \end{psmallmatrix}
  \in W^{1,2}(\Omega;\Com^2) : \ u_2 = i (n_1 + i n_2) u_1
  \mbox{ on } \partial\Omega
  \right\}
  .
\end{equation}
Under additional regularity conditions imposed on~$\Omega$,
it is known that~$H$ is self-adjoint, see
\cite{Benguria-Fournais-Stockmeyer-Bosch_2017b,
LeTreust-Ourmieres-Bonafos_2018,
Arrizibalaga-LeTreust-Raymond_2018}.

Since $W^{1,2}(\Omega;\Com^2)$ 
is compactly embedded in $\sii(\Omega;\Com^2)$,
the spectrum of~$H$ is purely discrete, 
composed of isolated eigenvalues of finite multiplicity
which accumulate at $\pm \infty$.
The spectrum is symmetric with respect to zero
and zero is never an eigenvalue.
We can therefore arrange the eigenvalues of~$H$
as follows 
\begin{equation*} 
  -\infty \leftarrow \dots \leq 
  -\lambda_3(\Omega) \leq -\lambda_2(\Omega) \leq - \lambda_1(\Omega)
  < 0 <
  \lambda_1(\Omega) \leq \lambda_2(\Omega) \leq \lambda_3(\Omega) 
  \leq \dots \to +\infty 
  \,,
\end{equation*}
where each eigenvalue is repeated according to its multiplicity. 
We are interested in the lowest positive energy $\lambda_1(\Omega)$.

\begin{Conjecture}\label{Conj}
For every $m \geq 0$,
\begin{enumerate}
\item[\emph{(i)}]
$\lambda_1(\Omega) \geq \lambda_1(\Omega^*)$,
where $\Omega^*$ is the disk of the same area as~$\Omega$,
\item[\emph{(ii)}]
$\lambda_1(\Omega) \geq \lambda_1(\Omega^*)$,
where $\Omega^*$ is the disk of the same perimeter as~$\Omega$.
\end{enumerate}
\end{Conjecture}

Part~(i) of the conjecture represents
the relativistic analogue of the Faber--Krahn inequality~\eqref{FK}. 
For massless particles (\ie\ $m=0$), 
part~(i) is explicitly stated as a conjecture in 
\cite{Antunes-Benguria-Lotoreichik-Ourmires-Bonafos_2021}.
The proof of Conjecture~\ref{Conj}
was classified as a hot open problem in spectral geometry
during an AIM workshop in San Jose (USA) in 2019,
see \cite{AIM-2019}.
For recent attempts to prove the conjecture,
see
\cite{Benguria-Fournais-Stockmeyer-Bosch_2017,
Lotoreichik-Ourmieres_2019,
Antunes-Benguria-Lotoreichik-Ourmires-Bonafos_2021}. 

While the case of general domains remains open,
it is precisely the goal of this paper to demonstrate
that an apparently simpler version of Conjecture~\ref{Conj} 
re-formulated for the rectangles~\eqref{rectangle} 
is equally challenging. 
That is, we conjecture that $\lambda_1(a,b) := \lambda_1(\Omega_{a,b})$
is optimised by the square both for the area 
or perimeter constraints.
\begin{Conjecture}\label{Conj.main}
For every $m \geq 0$,
\begin{enumerate}
\item[\emph{(i)}]
$\lambda_1(a,a^{-1}) \geq \lambda_1(1,1)$ with any $a>0$,
\hfill (area constraint)
\item[\emph{(ii)}]
$\lambda_1(a,2-a) \geq \lambda_1(1,1)$
with any $a \in (0,2)$.
\hfill (perimeter constraint)
\end{enumerate}
\end{Conjecture}

In part~(i) (respectively, (ii))
we consider the class of rectangles of area equal to~$1$
(respectively, perimeter equal to~$4$),
but there is no loss of generality in this restriction,
for other values can be recovered by scaling.

As a matter of fact, 
motivated by known non-relativistic results,
we expect that the inequalities 
in Conjecture~\ref{Conj.main}
are strict unless $a=1$.
That is, the square is the only minimiser of 
the spectral-optimisation problem among 
all the rectangles of fixed area or perimeter.

In contrast to the classical non-relativistic inequalities,
the apparent simplicity of Conjecture~\ref{Conj.main} is only illusory.
Indeed, writing $H_{a,b}$ for the Dirac operator~\eqref{operator1}
with~\eqref{operator2},
the eigenvalues of~$H_{a,b}$ in~$\Omega_{a,b}$
are not known explicitly.
More specifically, the spectral problem cannot be solved 
by a separation of variables, 
which we demonstrate in Section~\ref{Sec.no}. 
In particular, the eigenvalues are not a sum of the eigenvalues
of the Dirac operator in an interval,
subject to the infinite mass boundary conditions,
see Section~\ref{Sec.well}.  

Since explicit formulae for the eigenvalues are not available, 
we attack Conjecture~\ref{Conj.main}
by a detour through
a variational formulation involving the square of~$H_{a,b}$,
which is formulated in Section~\ref{Sec.minimax}.
This approach enables us to establish the following
upper and lower bounds,
which are of independent interest.

\begin{Theorem}\label{Thm.bounds}
For every $m \geq 0$, one has
\begin{multline*} 
  \left(\frac{\pi}{a}\right)^2 
  \max\left\{ \frac{1}{1+(ma)^{-1}}, \frac{1}{2} \right\}^2
  + \left(\frac{\pi}{b}\right)^2
  \max\left\{ \frac{1}{1+(mb)^{-1}}, \frac{1}{2} \right\}^2
  %\\
  \ \leq \ \lambda_1(a,b)^2 - m^2 \ \leq \
  %\\
  \left(\frac{\pi}{a}\right)^2 + \left(\frac{\pi}{b}\right)^2
  \,.
\end{multline*}
\end{Theorem}

Note that the upper bound is just the Dirichlet eigenvalue
$\Lambda_1(a,b) := \Lambda_1(\Omega_{a,b})$.
Therefore, Theorem~\ref{Thm.bounds} particularly implies 
$$
  \lambda_1(a,b)^2 - m^2 
  \ \xrightarrow[m\to\infty]{} \
  \Lambda_1(a,b)
  \,.
$$
This is a non-relativistic limit which is well known 
to hold for smooth domains 
(\cf~\cite{Arrizibalaga-LeTreust-Raymond_2017}). 
From this perspective, the mass terms in the lower bound 
of Theorem~\ref{Thm.bounds} can be interpreted 
as a relativistic correction to (known) ground-state energies 
in non-relativistic rectangles.
Finding the asymptotic expansion of $\lambda_1(a,b)^2 - m^2$
as $m \to \infty$ constitutes an interesting open problem.

Based on Theorem~\ref{Thm.bounds}, we are able to establish
Conjecture~\ref{Conj.main} in certain asymptotic regimes.

\begin{Corollary}\label{Corol}
Conjecture~\ref{Conj.main}.(i) holds 
under any of the following extra hypotheses:
\begin{enumerate}
\item[\emph{(a)}]
$
  |a^2-4| > \sqrt{15}
$, 
\hfill (large eccentricity)
\item[\emph{(b)}]
$
\displaystyle
  m \, \left( \frac{1}{a^2}+a^2-2 \right)
  \geq 56
$.
\hfill (heavy masses)
\end{enumerate}
Conjecture~\ref{Conj.main}.(ii) holds 
under any of the following extra hypotheses:
\begin{enumerate}
\item[\emph{(a')}]
$
\displaystyle
  |a-1|^2 > \frac{9-\sqrt{33}}{8}
$,
\hfill (large eccentricity)
\item[\emph{(b')}]
$
\displaystyle
  m \, \left( \frac{1}{a^2}+\frac{1}{(2-a)^2}-2 \right)
  \geq 56
$.
\hfill (heavy masses)
\end{enumerate}
Moreover, 
under any of these extra hypotheses,
the corresponding inequalities 
in Conjecture~\ref{Conj.main} are strict.
\end{Corollary}

Theorem~\ref{Thm.bounds} and Corollary~\ref{Corol}
are established in Section~\ref{Sec.mass}.
The quantitative conditions are not the best one can deduce 
from Theorem~\ref{Thm.bounds}, but they are particularly simple to check;
see Section~\ref{Sec.mass} for alternative estimates.
It is easily verified that~(a') and~(b')
are weaker than~(a) and~(b), respectively.
In general, part~(ii) of Conjecture~\ref{Conj.main}
always follows as a consequence of part~(i);
see the proof of Theorem~\ref{Thm.idea}.
An obvious defect of conditions~(b) and~(b') is that
the critical mass ensuring the validity of Conjecture~\ref{Conj.main}
diverges as $a \to 1$. 

Unfortunately, we have not been able to prove 
Conjecture~\ref{Conj.main} in its full generality.
Nevertheless, in Section~\ref{Sec.area},
we establish a sufficient condition (Theorem~\ref{Thm.idea})
which guarantees its validity.
The former is formulated in terms of a symmetry
of the minimiser of a non-convex optimisation problem
(Conjecture~\ref{Conj.symmetry}),
which we believe is of independent interest.
The main ingredient in this approach 
are symmetries of the rectangles  
investigated in Section~\ref{Sec.symmetry}.

%----------------------------------------%
\section{No separable solutions available}\label{Sec.no}
%----------------------------------------%
%
Let us argue that the spectral problem 
for the relativistic Hamiltonian~$H_{a,b}$
cannot be solved explicitly.
More specifically, we shall show that the problem
does not admit solutions with separated variables.

Recalling~\eqref{operator1} and~\eqref{operator2}, 
the eigenvalue problem $H_{a,b} u=\lambda u$ 
is equivalent to the system
\begin{equation}\label{system} 
\left\{
\begin{aligned}
  -i(\partial_1-i\partial_2) u_2 &= (\lambda-m) u_1
  && \mbox{in} \quad \Omega_{a,b} \,,
  \\
  -i(\partial_1+i\partial_2) u_1 &= (\lambda+m) u_2
  && \mbox{in} \quad \Omega_{a,b} \,,
  \\
  u_2 &= -u_1 
  && \mbox{on} \quad \left(-\ademi,\ademi\right) \times \left\{\bdemi\right\} \,,
  \\
  u_2 &= u_1 
  && \mbox{on} \quad \left(-\ademi,\ademi\right) \times \left\{-\bdemi\right\} \,,
  \\
  u_2 &= iu_1 
  && \mbox{on} \quad \left\{\ademi\right\} \times \left(-\bdemi,\bdemi\right) \,,
  \\
  u_2 &= -iu_1 
  && \mbox{on} \quad \left\{-\ademi\right\} \times \left(-\bdemi,\bdemi\right) 
  \,.
\end{aligned}  
\right.
\end{equation}
As already mentioned in the introduction,
the spectrum of~$H_{a,b}$ is symmetric with respect to zero.
This is easily seen by noticing that
$
  u
  = \begin{psmallmatrix}
  u_1 \\ u_2 
  \end{psmallmatrix}
$ 
is an eigenfunction of~$H_{a,b}$
corresponding to an eigenvalue~$\lambda$ if, and only if,
$
 \begin{psmallmatrix}
  \bar{u}_2 \\ \bar{u}_1 
  \end{psmallmatrix}
$  
is an eigenfunction of~$H_{a,b}$ corresponding to an eigenvalue~$-\lambda$
(charge conjugation symmetry).
At the same time, we necessarily have $|\lambda| > m$.
This is best seen from 
the (non-trivial but straightforwardly derived by an integration by parts) 
formula 
\begin{equation*}%\label{non-trivial} 
  \|H_{a,b} u\|^2 = \|\nabla u\|^2 + m^2 \|u\|^2 + m \, \|\gamma u\|^2  
\end{equation*}
valid for every $u \in \Dom(H_{a,b})$,
where $\gamma: W^{1,2}(\Omega_{a,b};\Com^2) \to \sii(\partial\Omega_{a,b};\Com^2)$ 
denotes the Dirichlet trace. 
(Using the same symbol~$\|\cdot\|$ for the different norms
should not cause any confusion, because the topology is 
determined by the space in which the respective function lies.)
Now, if $\lambda \in [-m,m]$ 
is an eigenvalue of~$H_{a,b}$ with an eigenfunction~$u$, 
then $\nabla u = 0$ in $\Omega_{a,b}$,
so~$u$ is a constant spinor, 
but constants do not satisfy the boundary conditions of~\eqref{system} 
contained in $\Dom(H_{a,b})$, 
a contradiction. 
 
Expressing~$u_2$ from the second equation of~\eqref{system} 
and putting it to the first differential equation,
we arrive at the following problem 
for the Laplacian with Cauchy--Riemann oblique boundary conditions:
\begin{equation}\label{Laplacian} 
\left\{
\begin{aligned}
  -\Delta u_1 &= (\lambda^2-m^2) u_1
  && \mbox{in} \quad \Omega_{a,b} \,,
  \\
  -i(\partial_1+i\partial_2) u_1 &= -(\lambda+m) u_1 
  && \mbox{on} \quad \left(-\ademi,\ademi\right) \times \left\{\bdemi\right\} \,,
  \\
  -i(\partial_1+i\partial_2) u_1 &= (\lambda+m) u_1 
  && \mbox{on} \quad \left(-\ademi,\ademi\right) \times \left\{-\bdemi\right\} \,,
  \\
  -i(\partial_1+i\partial_2) u_1 &= i(\lambda+m) u_1  
  && \mbox{on} \quad \left\{\ademi\right\} \times \left(-\bdemi,\bdemi\right) \,,
  \\
  -i(\partial_1+i\partial_2) u_1 &= -i(\lambda+m) u_1 
  && \mbox{on} \quad \left\{-\ademi\right\} \times \left(-\bdemi,\bdemi\right) \,.
\end{aligned}  
\right.
\end{equation}
Expressing~$u_1$ 
from the first equation of~\eqref{system} 
and putting it to the second differential equation
yields a similar problem for~$u_2$. 

Now, let us assume that there exist functions 
$\varphi \in W^{2,2}((-\ademi,\ademi))$
and $\chi \in W^{2,2}((-\bdemi,\bdemi))$
verifying $u_1(x_1,x_2) = \varphi(x_1)\chi(x_2)$
and~\eqref{Laplacian}.
Differentiating the first two 
(respectively, the last two)
boundary conditions of~\eqref{Laplacian}
with respect to the first (respectively, second) variable, 
one deduces that there exist constants $\alpha_1,\beta_1 \in \Com$
(respectively, $\alpha_2,\beta_2 \in \Com$) such that
$\varphi(x_1) = \alpha_1 e^{\beta_1 x_1}$ 
(respectively, $\chi(x_2) = \alpha_2 e^{\beta_2 x_2}$). 
Putting these solutions back to the boundary conditions of~\eqref{Laplacian}
and using that $\alpha_1 \not= 0$ and $\alpha_2 \not= 0$
(to have a non-trivial~$u_1$),
one obtains that necessarily $\lambda = -m$, a contradiction. 
In summary, Cauchy--Riemann oblique boundary conditions
cannot be satisfied by non-trivial functions with separated variables.
A similar argument excludes the possibility that
the problem for~$u_2$ admits a separation of variables, too. 

%----------------------------------------%
\section{Relativistic particle in a box}\label{Sec.well}
%----------------------------------------%
%	
The lack of separation of variables is related to the fact
that $H_{a,b}$ cannot be written as a sum of 
two one-dimensional operators.
Indeed, consider the one-dimensional operator
\begin{equation}\label{operator.1D}
\begin{aligned}
  H_a	 &:= 
  \begin{pmatrix}
    m & -i \partial \\
    -i \partial & -m
  \end{pmatrix}
  \qquad \mbox{in} \qquad
  \sii\left(\left(-\ademi,\ademi\right);\Com^2\right)
  \,,
  \\
  \Dom(H_a) &:= \left\{
  \varphi \in W^{1,2}\left(\left(-\ademi,\ademi\right);\Com^2\right), \
  \varphi_2(\pm\ademi) = \pm i \varphi_1(\pm\ademi)
  \right\}
  \,,
\end{aligned}  
\end{equation}
which corresponds to the ``longitudinal'' part
of the spectral problem~\eqref{system}.
At the same time, consider
the unitarily equivalent variant $\tilde{H}_b := r^*H_br$ with 
$
  r :=
  \begin{psmallmatrix}
    i & 0 \\
    0 & 1
  \end{psmallmatrix}
$,
which corresponds to the ``transversal'' part~\eqref{system}.
If~$\varphi$ and~$\tilde\varphi$ are eigenfunctions 
of~$H_a$ and $\tilde{H}_b$, then $\varphi \otimes \tilde\varphi$
is not an eigenfunction of $H_{a,b}$.
To see this fact, 
notice that imposing the boundary conditions of~\eqref{system} 
on the vertical boundaries
\begin{equation*} 
  \partial_\parallel\Omega_{a,b}
  := \left[\left\{-\ademi\right\} \times \left(-\bdemi,\bdemi\right)\right]
  \cup
  \left[\left\{\ademi\right\} \times \left(-\bdemi,\bdemi\right)\right] 
\end{equation*}
implies that $\tilde\varphi_2=\mp i\tilde\varphi_1$
unless $\varphi(\pm\ademi)=0$.
At the same time, imposing the boundary conditions of~\eqref{system} 
on the horizontal boundaries
\begin{equation*} 
  \partial_=\Omega_{a,b}
  := \left[\left(-\ademi,\ademi\right) \times \left\{-\bdemi\right\}\right] 
  \cup 
  \left[\left(-\ademi,\ademi\right) \times \left\{\bdemi\right\}\right] 
\end{equation*}
implies that $\varphi_2=\mp\varphi_1$
unless $\tilde\varphi(\pm\bdemi)=0$.
In any case, $H_{a,b}$ would have to have an eigenfunction 
which satisfies the Dirichlet boundary condition
on $\partial_\parallel\Omega_{a,b}$ or $\partial_=\Omega_{a,b}$.
But then the differential equations of~\eqref{system}  
would imply that the eigenfunction satisfies also
the Neumann boundary condition on the same piece of boundary.
This would lead to an overdetermined problem, a contradiction.

A more direct way how to get the contradiction is to realise that
the spectral problem for~$H_a$ can be solved explicitly
in terms of sines and cosines (\cf~\cite[Sec.~2.1]{BBKO})
and that the eigenfunctions of~$H_a$ never vanish at~$\pm\ademi$. 
In particular, the lowest positive eigenvalue~$\lambda_1(a)$ of~$H_a$
equals 
$$
  \lambda_1(a) = \sqrt{m^2 + \left(\frac{\nu_1(ma)}{a}\right)^2}
  \,,
$$
where $\nu_1(m)$ is the unique root of the equation
\begin{equation}\label{root}
  \frac{\tan(\nu)}{\nu} = \frac{-1}{m}
\end{equation}
lying in the interval $\left[\frac{\pi}{2},\pi\right)$.
In fact, $\nu_1(0) = \frac{\pi}{2}$ and $\nu_1(m) \to \pi$ as $m \to \infty$.
It follows that $\lambda_1(a)^2 - m^2$ converges to 
the lowest eigenvalue $\Lambda_1\left(\left(-\ademi,\ademi\right)\right)$
of the Dirichlet Laplacian in 
$\sii\left(\left(-\ademi,\ademi\right)\right)$.
Applying the inequality $\tan(\nu) \leq \nu-\pi$ 
for $\nu \in \left[\frac{\pi}{2},\pi\right)$
to~\eqref{root} shows
\begin{equation}\label{lower}
  \nu(m) \geq \pi \, \max\left\{ \frac{1}{1+m^{-1}}, \frac{1}{2} \right\}
\end{equation}
for every $m \geq 0$
(the number $(1+m^{-1})^{-1}$ is interpreted as zero for $m=0$).
While the estimate is not particularly good for small masses 
(though trivially sharp for $m=0$),
it is a good approximation for larges masses
(and asymptotically sharp in the limit $m \to \infty$).

%--------------------------------%
\section{Variational formulation}\label{Sec.minimax}
%--------------------------------%
%
Since explicit solutions of the eigenvalue problem for~$H_{a,b}$
are not available,
we attack Conjecture~\ref{Conj.main}
by a detour through
the variational formulation 
\begin{equation*}  
  \lambda_1(a,b)^2 
  = \inf_{\stackrel[u \not= 0]{}{u \in \Dom(H_{a,b})}} 
  \frac{\|H_{a,b} u\|^2}{\|u\|^2}  
  \,.
\end{equation*}
Indeed, the right-hand side is just the standard Rayleigh--Ritz 
variational formula 
for the lowest eigenvalue of the square~$H_{a,b}^2$
(see, \eg, \cite[Sec.~4.5]{Davies}).  
It remains to notice that the latter 
equals the square of the eigenvalue of~$H_{a,b}$
which is closest to zero and recall the symmetry of 
the spectrum of~$H_{a,b}$.

It will be useful to work in a Hilbert space
independent of the parameters $a,b$.
More specifically, we introduce the unitary transform
$
  U: \sii(\Omega_{a,b};\Com^2)
  \to \sii(\Omega_{1,1};\Com^2)
$
by setting $(Uu)(x):=\sqrt{ab} \, u(ax_1,bx_2)$
and define a unitarily equivalent (therefore isospectral) operator
$\hat{H}_{a,b} := U H_{a,b} U^{-1}$.
Clearly, $\hat{H}_{1,1} = H_{1,1}$.
Moreover, $\Dom(\hat{H}_{a,b})=\Dom(\hat{H}_{1,1})$ for every $a,b>0$,
so we actually have $\Dom(\hat{H}_{a,b}) = \Dom(H_{1,1})$.

Denoting $\psi := Uu$,
one has
\begin{equation}\label{Rayleigh} 
  \lambda_1(a,b)^2 
  = \inf_{\stackrel[\psi \not= 0]{}{\psi \in \Dom(H_{1,1})}} 
  \frac{\|\hat{H}_{a,b}\psi\|^2}{\, \|\psi\|^2}  
\end{equation}
with 
\begin{equation}\label{hat} 
  \|\hat{H}_{a,b}\psi\|^2 :=
  \frac{1}{a^2} \, \|\partial_1\psi\|^2 
  + \frac{1}{b^2} \, \|\partial_2\psi\|^2
  + m^2 \|\psi\|^2 
  + \frac{m}{a} \, \|\gamma_\parallel\psi\|^2
  + \frac{m}{b} \, \|\gamma_=\psi\|^2
  \,,
\end{equation}
where $\gamma_=$ (respectively, $\gamma_\parallel$) 
stands for the trace operator on 
the horizontal boundary $\partial_=\Omega_{1,1}$
(respectively, the vertical boundary $\partial_\parallel\Omega_{1,1}$).

Using the analogous variational formulation 
for $\lambda_1(a)$ of the one-dimensional operator~\eqref{operator.1D},
one particularly gets the Poincar\'e-type inequality
\begin{equation}\label{Poincare}
  \frac{1}{a^2} \, \|\varphi'\|^2 
  + \frac{m}{a} \, \left(|\varphi(-\ademi)|^2 + |\varphi(\ademi)|^2\right)
  \geq \left(\frac{\nu_1(ma)}{a}\right)^2 \|\varphi\|^2
\end{equation}
valid for every $\varphi \in \Dom(H_a)$.

%--------------------------------------% 
\section{Proofs}\label{Sec.mass}
%--------------------------------------% 
% 
Now we have all the ingredients to establish
Theorem~\ref{Thm.bounds} and its Corollary~\ref{Corol}.

\begin{proof}[Proof of Theorem~\ref{Thm.bounds}]
The upper bound follows by using the Dirichlet eigenfunction
\begin{equation*}%\label{Dirichlet.ef} 
  \psi_D(x_1,x_2) := \cos(\pi x_1) \cos(\pi x_2) 
  \begin{pmatrix}
    C_1 \\ C_2
  \end{pmatrix}
  ,
\end{equation*}
where $C_1,C_2$ are arbitrary complex numbers
not simultaneously equal to zero,
as a trial function in~\eqref{Rayleigh}.
The lower bound is a consequence of the better estimate
$$
\begin{aligned}
  \lambda_1(a,b)^2 - m^2
  \geq \left(\frac{\nu_1(ma)}{a}\right)^2 
  + \left(\frac{\nu_1(mb)}{b}\right)^2 
  \,,
\end{aligned}  
$$
which follows from using Fubini in~\eqref{hat}
followed by applying the 1-dimensional
Poincare inequality~\eqref{Poincare} in each variable. 
Then the lower bound in Theorem~\ref{Thm.bounds}
follows from applying the crude bound~\eqref{lower}.  
\end{proof}
\begin{proof}[Proof of Corollary~\ref{Corol}]
For the area constraint we take $b:=1/a$
and allow~$a$ to be arbitrary,  
so that $|\Omega_{a,b}| = |\Omega_{1,1}| = 1$. 
For the perimeter constraint we take $b:=2-a$ 
and restrict ourselves to $a \in (0,2)$, 
so that $|\partial\Omega_{a,b}| = |\partial\Omega_{1,1}| = 4$. 

The lower bound of Theorem~\ref{Thm.bounds} particularly implies 
$$
  \lambda_1(a,b)^2-m^2 
  \geq \frac{\pi^2}{4} \left(\frac{1}{a^2}+\frac{1}{b^2}\right)
  \,.
$$
Requiring that this lower bound is strictly greater than
$
  2\pi^2 \geq \lambda_1(1,1)^2-m^2 
$,
where the inequality follows from
the upper bound of Theorem~\ref{Thm.bounds},
leads immediately to the conditions
$$
  \frac{1}{a^2} + a^2 > 8
  \qquad \mbox{and} \qquad
  \frac{1}{a^2}+\frac{1}{(2-a)^2} > 8
$$
in the area and perimeter constraint, respectively.
These inequalities are equivalent to 
conditions~(a) and~(a'), respectively.
%The latter is implied by the stronger requirement
%$
%  (2-a)^2 + a^2 > 8 a(2-a)
%$,
%which is equivalent to condition~(a').
For further purposes, let us observe that~(a) 
(respectively, (a')) is implied by $a\geq 3$ or $a \leq 1/3$
(respectively, $a \geq 5/3$ or $a \leq 1/3$).

For the other pair of conditions, 
we estimate the lower bound of Theorem~\ref{Thm.bounds} as follows:
$$
\begin{aligned}
  \lambda_1(a,b)^2-m^2 
  &\geq 
  \left(\frac{\pi}{a}\right)^2 
  \left(\frac{1}{1+(ma)^{-1}}\right)^2
  + \left(\frac{\pi}{b}\right)^2
  \left(\frac{1}{1+(mb)^{-1}} \right)^2
  \\
  &\geq
  \left(\frac{\pi}{a}\right)^2 
  \left(1-\frac{2}{ma}\right)
  + \left(\frac{\pi}{b}\right)^2
  \left(1-\frac{2}{mb} \right)
  \,,
\end{aligned}  
$$
where the second inequality employs 
the convexity of $z \mapsto (1+z)^{-2}$ at $z=0$.
Requiring that this lower bound is strictly greater than
$
  2\pi^2 \geq \lambda_1(1,1)^2-m^2 
$,
where the inequality follows from
the upper bound of Theorem~\ref{Thm.bounds},
leads immediately to the conditions
\begin{equation}\label{initial}
  m \, \frac{\displaystyle \frac{1}{a^2}+a^2-2}
  {\displaystyle \frac{1}{a^3}+a^3}
  > 2
  \qquad \mbox{and} \qquad
   m \, \frac{\displaystyle \frac{1}{a^2}+\frac{1}{(2-a)^2}-2}
  {\displaystyle \frac{1}{a^3}+\frac{1}{(2-a)^3}}
  > 2
\end{equation}  
in the area and perimeter constraint, respectively.
By virtue of~(a) (respectively, (a')), 
we may restrict ourselves to $a \in (1/3,3)$ 
(respectively, $a \in (1/3,5/3)$).
Using these restrictions in the denominators
of~\eqref{initial} as follows 
$$
  \frac{1}{a^3}+a^3 < \frac{1}{27} + 27 < 28
  \qquad \mbox{and} \qquad
  \frac{1}{a^3}+\frac{1}{(2-a)^3} < \frac{3\,402}{125} < 28
  \,,
$$
we arrive at~(b) and~(b').
\end{proof}

The rest of the paper presents an attempt to prove
Conjecture~\ref{Conj.main} without any extra hypotheses.

%-------------------%
\section{Symmetries}\label{Sec.symmetry}
%-------------------%
%
%Our proof of Theorem~\ref{Thm}
%is based on symmetry properties of the rectangles.
For any $u \in \sii(\Omega_{1,1};\Com^2)$, let us introduce 
the transformed spinor
\begin{equation*} 
  (Ru)(x_1,x_2) := 
  \begin{pmatrix}
     i u_1(-x_2,x_1) \\
     u_2(-x_2,x_1) 
  \end{pmatrix}  
  \,.
\end{equation*}
The action of~$R$ can be interpreted as a rotation by $90$ degrees.
Then the following result can be understood as a well-known
symmetry of the general rectangles $\Omega_{a,b}$ 
with respect to rotations by $180$ degrees.
\begin{Lemma}\label{Lem.symmetry0}
For any eigenvalue~$\lambda$ of $\hat{H}_{a,b}$,
there exists an eigenfunction $\psi \in \Dom(\hat{H}_{a,b})$ 
satisfying
\begin{equation}\label{beta} 
  R^2\psi = \beta\psi 
  \qquad \mbox{with} \qquad
  \beta \in \{\pm 1\}
  \,.
\end{equation}
\end{Lemma}
\begin{proof}
It is straightforward to check that 
if $u \in \Dom(\hat{H}_{a,b})$ solves $\hat{H}_{a,b} u = \lambda u$,
then $R^2u \in \Dom(\hat{H}_{a,b})$ 
and $\hat{H}_{a,b} R^2u = \lambda R^2u$.
At the same time,
it is easy to see that~$R^2u$ is non-zero if, 
and only if, $u$~is non-zero;
in fact, $\|R^2u\|=\|u\|$.

Let $(\psi_{(1)},\dots,\psi_{(N)})$ 
with a positive integer~$N$
be any basis of the kernel
$\Ker(\hat{H}_{a,b}-\lambda)$.
The rotation~$R^2$ can be considered as an operator on this eigenspace.
Let~$B$ denote the matrix of~$R^2$ with respect to 
the eigenbasis; more specifically,
$B:=(b_{jk})$ with $j,k \in \{1,\dots,N\}$,
where $b_{jk}$'s are the coefficients in the decompositions
$
  R^2\psi_{(k)} = b_{1k} \psi_{(1)} 
  + b_{2k} \psi_{(2)} 
  + \dots + b_{Nk} \psi_{(N)}
$.
Since $R^4u = u$, one has $B^2 = I$.
Consequently, $\sigma(B) \subset \{\pm 1\}$. 
Given any eigenvalue $\beta \in \sigma(B)$,
let $c \in \Com^N$ be a corresponding eigenvector.
Let us define 
\begin{equation}\label{eigenvector}
  \psi := c_1 \psi_{(1)} + \dots + c_N \psi_{(N)}
  \,,
\end{equation}
which is necessarily non-zero.
Then
\begin{equation*} 
  R^2\psi = \sum_{k=1}^N c_k R^2\psi_{(k)}
  = \sum_{j,k=1}^N c_k b_{jk} \psi_{(j)}
  = \sum_{j,k=1}^N \beta c_k \delta_{jk} \psi_{(j)}
  = \beta\psi
  \,.
\end{equation*}
This concludes the proof of the lemma.
\end{proof}

It is not surprising that the squares~$\Omega_{a,a}$ admit 
a higher degree of symmetry.
\begin{Lemma}\label{Lem.symmetry}
Given any eigenvalue~$\lambda$ of $\hat{H}_{a,a}$,
there exists an eigenfunction $\psi \in \Dom(\hat{H}_{a,a})$ 
satisfying
\begin{equation*} 
  R\psi = \alpha\psi 
  \qquad \mbox{with} \qquad
  \alpha \in \{\pm 1,\pm i\}
  \,.
\end{equation*}
\end{Lemma}
\begin{proof}
The proof follows the lines of the proof of Lemma~\ref{Lem.symmetry0}.
Again, it is straightforward to check that 
if $u \in \Dom(\hat{H}_{a,a})$ solves $\hat{H}_{a,a} u = \lambda u$,
then $Ru \in \Dom(\hat{H}_{a,a})$ 
and $\hat{H}_{a,a} Ru = \lambda Ru$.
At the same time,
it is easy to see that~$Ru$ is non-zero if, 
and only if, $u$~is non-zero;
in fact, $\|Ru\|=\|u\|$.

Let $(\psi_{(1)},\dots,\psi_{(N)})$ 
with a positive integer~$N$
be any basis of the kernel
$\Ker(\hat{H}_{a,a}-\lambda)$.
The rotation~$R$ can be considered as an operator on this eigenspace.
Let~$A$ denote the matrix of~$R$ with respect to 
the eigenbasis; more specifically,
$A:=(a_{jk})$ with $j,k \in \{1,\dots,N\}$,
where $a_{jk}$'s are the coefficients in the decompositions
$
  R\psi_{(k)} = a_{1k} \psi_{(1)} 
  + a_{2k} \psi_{(2)} 
  + \dots + a_{Nk} \psi_{(N)}
$.
Since $R^4u = u$, one has $A^4 = I$.
Consequently, $\sigma(A) \subset \{\pm 1,\pm i\}$. 
Given any eigenvalue $\alpha \in \sigma(A)$,
let $c \in \Com^N$ be a corresponding eigenvector.
Let us define~$\psi$ as in~\eqref{eigenvector},
which is necessarily non-zero.
Then
\begin{equation*}%\label{rotation.check}
  R\psi = \sum_{k=1}^N c_k R\psi_{(k)}
  = \sum_{j,k=1}^N c_k a_{jk} \psi_{(j)}
  = \sum_{j,k=1}^N \alpha c_k \delta_{jk} \psi_{(j)}
  = \alpha\psi
  \,.
\end{equation*}
This concludes the proof of the lemma.
\end{proof}

The hypotheses of the following lemma 
are particularly verified for the symmetric 
eigenfunctions of the square~$\Omega_{1,1}$ 
due to Lemma~\ref{Lem.symmetry}.
\begin{Lemma}\label{Lem.Corol}
Let $\psi \in \Dom(H_{1,1})$ satisfy 
\begin{equation}\label{omega}
  R\psi=\omega\psi
  \qquad\mbox{with some }
  \omega \in \Com
  \mbox{ such that } |\omega|=1
  \,.
\end{equation}
Then
\begin{equation}\label{symmetry}
  \|\partial_1\psi\| = \|\partial_2\psi\|
  \qquad \mbox{and} \qquad
  \|\gamma_\parallel\psi\| = \|\gamma_=\psi\|
  \,.
\end{equation}
\end{Lemma}
\begin{proof}
One has
\begin{equation*}
  (\partial_1\psi)(x_1,x_2)
  = 
  \begin{pmatrix}
     \partial_{x_1} [\psi_1(x_1,x_2)] \\
     \partial_{x_1} [\psi_2(x_1,x_2)] 
  \end{pmatrix}
  = \frac{1}{\omega} 
  \begin{pmatrix}
     i \partial_{x_1} [\psi_1(-x_2,x_1)] \\
     \partial_{x_1} [\psi_2(-x_2,x_1)] 
  \end{pmatrix}
  = \frac{1}{\omega} 
  \begin{pmatrix}
     i (\partial_{2} \psi_1)(-x_2,x_1) \\
     (\partial_{2} \psi_2)(-x_2,x_1) 
  \end{pmatrix}
  \,.
\end{equation*}
Consequently,
\begin{equation*}
\begin{aligned}
  \|\partial_1\psi\|^2
  &= \frac{1}{|\omega|^2} \int_{\Omega_{1,1}}  
  \left(
  |i (\partial_{2} \psi_1)(-x_2,x_1)|^2 
  + |(\partial_{2}\psi_2)(-x_2,x_1)|^2
  \right) \der x_1 \, \der x_2
  \\
  &= \int_{\Omega_{1,1}}  
  \left(
  |(\partial_{2} \psi_1)(x_1,x_2)|^2 
  + |(\partial_{2}\psi_2)(x_1,x_2)|^2
  \right) \der x_1 \, \der x_2
  \\
  &= \|\partial_2\psi\|^2
  \,,
\end{aligned}  
\end{equation*}
where the second equality follows by an obvious change of variables.
This establishes the first identity of~\eqref{symmetry}.
At the same time, 
\begin{equation*}
\begin{aligned}
  \|\gamma_\parallel\psi\|^2
  &= \int_{-\frac{1}{2}}^{\frac{1}{2}}  
  \left(
  |\psi_1(-\demi,x_2)|^2 + |\psi_2(-\demi,x_2)|^2
  + |\psi_1(\demi,x_2)|^2 + |\psi_2(\demi,x_2)|^2
  \right)
  \der x_2 
  \\
  &= \frac{1}{|\omega|^2} 
  \int_{-\frac{1}{2}}^{\frac{1}{2}}  
  \left(
  |i\psi_1(-x_2,-\demi)|^2 + |\psi_2(-x_2,-\demi)|^2
  + |i\psi_1(-x_2,\demi)|^2 + |\psi_2(-x_2,\demi)|^2
  \right)
  \der x_2 
  \\
  &=  
  \int_{-\frac{1}{2}}^{\frac{1}{2}}  
  \left(
  |\psi_1(x_1,-\demi)|^2 + |\psi_2(x_1,-\demi)|^2
  + |\psi_1(x_1,\demi)|^2 + |\psi_2(x_1,\demi)|^2
  \right)
  \der x_1
  \\
  &= \|\gamma_=\psi\|^2 
  \,,
\end{aligned}   
\end{equation*}
which establishes the second identity of~\eqref{symmetry}.
\end{proof}
%
 
%--------------------------------------% 
\section{Non-convex optimisation}\label{Sec.area}
%--------------------------------------% 
% 
Now, let us take $b:=1/a$,
so that $|\Omega_{a,b}| = |\Omega_{1,1}| = 1$.
Recalling~\eqref{Rayleigh} with~\eqref{hat}, 
one has
\begin{equation}\label{Rayleigh.area} 
  \lambda_1(a,a^{-1})^2 - m^2
  = \inf_{\stackrel[\psi \not= 0]{}
  {\psi \in \Dom(H_{1,1})}} 
  \frac{\displaystyle a^{-2} \, \|\partial_1\psi\|^2 
  + a^2 \, \|\partial_2\psi\|^2
  + m a^{-1} \, \|\gamma_\parallel\psi\|^2
  + m a \, \|\gamma_=\psi\|^2}
  {\, \|\psi\|^2}  
  \,.
\end{equation}
Using the elementary inequality
$p^2+q^2 \geq 2pq$
valid for all real numbers~$p$ and~$q$,
we get
\begin{equation}\label{area} 
  \lambda_1(a,a^{-1})^2 - m^2
  \geq \inf_{\stackrel[\psi \not= 0]{}
  {\psi \in \Dom(H_{1,1})}} 
  \frac{2\,\|\partial_1\psi\| \|\partial_2\psi\|
  +2m \, \|\gamma_\parallel\psi\| \|\gamma_=\psi\|}
  {\, \|\psi\|^2}  
  =: \mu
  \,.
\end{equation}

The minimisation problem on the right-hand side of~\eqref{area}
does not involve a convex functional.
In fact, the associated Euler equation is a non-linear problem.
Recalling that we use the same symbol $\|\cdot\|$
for norms in different spaces,
here $(\cdot,\cdot)$ stands for respective inner products.

\begin{Lemma}\label{Lem.minimiser}
The infimum on the right-hand side of~\eqref{area} is achieved.
Any minimiser~$\psi$ satisfies 
the weak eigenvalue equation 
\begin{equation}\label{Euler}
  A^{-2} \, (\partial_1\phi,\partial_1\psi)
  + A^2 \, (\partial_2\phi,\partial_2\psi)
  + m B^{-1} \, (\gamma_\parallel\phi,\gamma_\parallel\psi)
  + m B \, (\gamma_=\phi,\gamma_=\psi)
  = \mu \, (\phi,\psi)
\end{equation} 
for every $\phi \in \Dom(H_{1,1})$,
where
$$
  A := \sqrt{\frac{\|\partial_1\psi\|}{\|\partial_2\psi\|}}
  \qquad\mbox{and}\qquad
  B := \frac{\|\gamma_\parallel\psi\|}{\|\gamma_=\psi\|}
  \,.
$$
\end{Lemma}
\begin{proof}
First of all, let us notice that 
there exists a positive constant~$c$ such that 
\begin{equation}\label{2D}
  \forall \psi \in \Dom(H_{1,1}) \,, \qquad
  \|\partial_1\psi\| \geq c \, \|\psi\|
  \qquad\mbox{and}\qquad
  \|\partial_2\psi\| \geq c \, \|\psi\|
  \,.
\end{equation}
Indeed, with help of~\eqref{Poincare} and Fubini's theorem,
one can take $c := \nu_1(0) = \frac{\pi}{2}$ 
introduced in~\eqref{root}.
Now, let us argue that the infimum in~\eqref{area} is indeed achieved.
Define the functional
\begin{equation*} 
  J[\psi] := 2\,\|\partial_1\psi\| \|\partial_2\psi\|
  +2m \, \|\gamma_\parallel\psi\| \|\gamma_=\psi\|
  \,.
\end{equation*}
Then 
\begin{equation}\label{area.bis}  
  \mu = \inf_{\stackrel[\|\psi\| = 1]{}
  {\psi \in \Dom(H_{1,1})}} 
  J[\psi] 
  \,.
\end{equation}
Let $\{\psi_j\}_{j \in \Nat}$ be a minimising sequence,
\ie\ $J[\psi_j] \to \mu$ as $j \to \infty$ 
and $\|\psi_j\|=1$ for every $j \in \Nat$.
Consequently,  
\begin{equation*}%\label{bound} 
  c \, \|\nabla\psi_j\|
  \leq c \left(\|\partial_1\psi_j\| + \|\partial_2\psi_j\| \right)
  \leq 2 \, \|\partial_1\psi_j\| \|\partial_2\psi_j\| 
  \leq J[\psi_j] = \mu
\end{equation*}
for every $j \in \Nat$,
where the second inequality is due to~\eqref{2D}
and that $\|\psi_j\|=1$.
It follows that $\{\psi_j\}_{j \in \Nat}$
is a bounded sequence in $W^{1,2}(\Omega_{1,1};\Com^2)$.
Therefore, up to a subsequence, 
$\{\psi_j\}_{j \in \Nat}$ converges weakly to some~$\psi$ 
in $W^{1,2}(\Omega_{1,1};\Com^2)$.
By the compactness of the embedding
$W^{1,2}(\Omega_{1,1};\Com^2)$
in $\sii(\Omega_{1,1};\Com^2)$,
we may assume that $\{\psi_j\}_{j \in \Nat}$ converges (strongly) 
to some~$\psi$ in $\sii(\Omega_{1,1};\Com^2)$
such that $\|\psi\|=1$.
By using~$\psi$ as test function in~\eqref{area.bis}, 
we obviously have $\mu \leq J[\psi]$. 
On the other hand, 
\begin{equation*}
  \mu = \liminf_{j\to \infty} J[\psi_j] 
  \geq J[\psi]
  \,,
\end{equation*}
where the inequality follows by the property that~$J$
is lower semicontinuous.
In summary, $\mu = J[\psi]$, 
so the infimum in~\eqref{area.bis} can be replaced by a minimum.

Now, let~$\psi$ be any minimiser of~\eqref{area.bis}.
Then~$\psi$ is a critical point of the functional~$J$
and the derivative  
$$
  \lim_{\eps \to 0}
  \frac{1}{\eps}
  \left(
  \frac{J[\psi+\eps\phi]}{\|\psi+\eps\phi\|^2}
  - \frac{J[\psi]}{\|\psi\|^2}
  \right)
$$
is necessarily equal to zero for any choice 
of the test function $\phi \in \Dom(H_{1,1})$.
This leads to the equation
\begin{equation*}
  A^{-2} \, \Re(\partial_1\phi,\partial_1\psi)
  + A^2 \, \Re(\partial_2\phi,\partial_2\psi)
  + m B^{-1} \, \Re(\gamma_\parallel\phi,\gamma_\parallel\psi)
  + m B \, \Re(\gamma_=\phi,\gamma_=\psi)
  = \mu \, \Re(\phi,\psi)
  \,.
\end{equation*}
Combining this equation with its variant where~$\phi$ is replaced by $i\phi$,
it is clear that the real part can be removed,
so~\eqref{Euler} follows. 
 
Finally, let us argue that~\eqref{Euler} is well defined,
meaning that~$A$ and~$B$ are positive and bounded. 
If $\|\partial_1\psi\|=0$, 
then $\psi$ is independent of the first variable,
which is incompatible with $\psi \in \Dom(H_{1,1})$
(\cf~the boundary conditions of~\eqref{system})
unless $\psi = 0$ identically. 
An analogous argument excludes the possibility $\|\partial_2\psi\|=0$.
If $\|\gamma_=\psi\|=0$, 
then~$\psi$ satisfies~\eqref{Euler} with $m=0$,
subject to an extra Dirichlet boundary condition on 
$\partial_=\Omega_{1,1}$. 
More specifically, 
it follows from~\eqref{Euler}
by standard elliptic regularity that
$\psi$~belongs to
$  
  W^{2,2}(\Omega_\eps;\Com^2\big)
$
and solves 
\begin{equation}\label{elliptic}
  (-A^{-2} \, \partial_1^2 - A^2 \, \partial_2^2)\psi = \mu \psi
  \qquad\mbox{in}\qquad
  \Omega_\eps := (-\demi+\eps,\demi-\eps)\times (-\demi,\demi) 
  \subset \Omega_{1,1}
  \,,
\end{equation}
where $\eps \in (0,\demi)$ is arbitrary,
subject to boundary conditions 
\begin{equation}\label{Dirichlet.bc}
  \partial_2\psi_2 = \pm\partial_2\psi_1
  \quad \mbox{and} \quad 
  \psi_1=0=\psi_2
  \qquad \mbox{on} \qquad 
  \Gamma_\eps^\pm :=
  \left(-\demi+\eps,\demi-\eps\right) \times \left\{\pm\demi\right\} 
  \,.
\end{equation}
Here the former boundary condition follows
from~\eqref{Euler} when using the arbitrariness of $\phi \in \Dom(H_{1,1})$,
while the latter is due to $\|\gamma_=\psi\|=0$.
It is the latter which plays a crucial role in our argument below
(the former will not be used).
In order to obtain a contradiction, 
it suffices to show that~$\psi$ satisfies 
both Dirichlet and Neumann conditions on $\Gamma_\eps^\pm$.
Indeed, 
using~$\psi$ as a trial function in~\eqref{Rayleigh.area} with $m=0$,
one has
\begin{equation}\label{testing}
  \lambda_1(A,A^{-1})^2 
  \leq \frac{A^{-2} \|\partial_1\psi\|^2 + A^{2} \|\partial_2\psi\|^2}
  {\|\psi\|^2}
  = \mu
  \,.
\end{equation}
Combining this upper bound with~\eqref{area},
it follows that $\lambda_1(A,A^{-1})^2=\mu$,
so~$\psi$ is also a minimiser of~\eqref{Rayleigh.area} with $m=0$.
Then~$\psi$ solves a properly rescaled problem~\eqref{system} 
with $m=0$, $a=A$, $b=A^{-1}$ 
and $\lambda = \lambda_1(A,A^{-1})$ or $\lambda = -\lambda_1(A,A^{-1})$.
Irrespectively of the sign of~$\lambda$, 
the differential equations of~\eqref{system}  
with help of the Dirichlet boundary condition of~\eqref{Dirichlet.bc}
imply that 
\begin{equation}\label{Neumann.bc}
  \partial_2\psi_1 = 0 = \partial_2\psi_2 
  \qquad \mbox{on} \qquad 
  \Gamma_\eps^\pm \,.
\end{equation}
Hence~$\psi$ satisfies the elliptic equation~\eqref{elliptic},
subject to Dirichlet~\eqref{Dirichlet.bc} 
and Neumann~\eqref{Neumann.bc} boundary conditions
imposed simultaneously on~$\Gamma_\eps^\pm$.
This is an overdetermined problem,
so necessarily $\psi=0$ identically in~$\Omega_\eps$.
Since~$\eps$ can be made arbitrarily small,
it implies that $\psi=0$ identically in~$\Omega_{1,1}$,
a contradiction.  
Consequently, $\|\gamma_=\psi\|\not=0$.
An analogous argument excludes 
the possibility $\|\gamma_\parallel\psi\|=0$.
\end{proof}

The following conjecture is naturally expected.
\begin{Conjecture}\label{Conj.symmetry}
There exists a minimiser~$\psi$
of the right-hand side of~\eqref{area}
which satisfies~\eqref{symmetry}. 
\end{Conjecture}

We have not been able to establish this conjecture.
This is unfortunate, because its validity  
immediately implies Conjecture~\ref{Conj.main}.
\begin{Theorem}\label{Thm.idea}
Conjecture~\ref{Conj.symmetry} implies Conjecture~\ref{Conj.main}.
\end{Theorem}
\begin{proof}
As a consequence of~\eqref{area} 
and Conjecture~\ref{Conj.symmetry},
\begin{equation}\label{comp1}  
  \lambda_1(a,a^{-1})^2 - m^2
  \geq \inf_{\stackrel[\psi \not= 0]{}
  {\psi \in \Dom(H_{1,1}) \ \& \ \eqref{symmetry} \text{ holds}}} 
  \frac{2\,\|\partial_1\psi\| \|\partial_2\psi\|
  +2m \, \|\gamma_\parallel\psi\| \|\gamma_=\psi\|}
  {\, \|\psi\|^2}  
  \,.
\end{equation}
At the same time,
as a consequence of Lemmata~\ref{Lem.symmetry} and~\ref{Lem.Corol},
\begin{equation}\label{comp2} 
  \lambda_1(1,1)^2 - m^2
  = \inf_{\stackrel[\psi \not= 0]{}
  {\psi \in \Dom(H_{1,1}) \ \& \ \eqref{symmetry} \text{ holds}}} 
  \frac{2\,\|\partial_1\psi\| \|\partial_2\psi\|
  +2m \, \|\gamma_\parallel\psi\| \|\gamma_=\psi\|}
  {\, \|\psi\|^2}  
  \,.
\end{equation}
Comparing~\eqref{comp1} and~\eqref{comp2}, 
we get part~(i) of Conjecture~\ref{Conj.main}.

For the perimeter constraint, 
let us take $b:=2-a$ and restrict ourselves to $a \in (0,2)$, 
so that $|\partial\Omega_{a,b}| = |\partial\Omega_{1,1}| = 4$. 
Recalling~\eqref{Rayleigh} with~\eqref{hat}
and noticing that 
$(2-a)^{-1} \geq a$ for every $a \in (0,2)$,
it follows that
\begin{equation*} 
  \lambda_1(a,2-a) 
  \geq \lambda_1(a,a^{-1}) 
  \geq \lambda_1(1,1) 
  \,,
\end{equation*}
where the last inequality is due to part~(i) of Conjecture~\ref{Conj.main}.
Hence, part~(ii) follows as a consequence of~(i).
\end{proof}
\begin{Remark}
Let us present a naive approach 
to establish a weaker variant Conjecture~\ref{Conj.symmetry},
which would be sufficient for the proof of Conjecture~\ref{Conj.main}
(along the lines of the proof of Theorem~\ref{Thm.idea}).
The idea is to employ Lemma~\ref{Lem.symmetry0}
and add the constraint~\eqref{beta} 
to the minimisation problem on the right-hand side of~\eqref{area}. 
Let~$\psi$ be any such minimiser.
We wish to show that it satisfies~\eqref{omega}
(which then implies the desired symmetry~\eqref{symmetry}
by Lemma~\ref{Lem.Corol}).

If $R\psi=\omega \psi$ with some $\omega \in \Com$ such that $|\omega|=1$,
then~$\psi$ is the desired minimiser.
Let us therefore assume that $R\psi - \omega \psi \not= 0$
for every $\omega \in \Com$ such that $|\omega|=1$.
If $\beta=1$ (\ie\ $R^2\psi = \psi$), then the function 
$
  v := \psi + R\psi + R^2 \psi + R^3 \psi = 2(\psi+R\psi) 
$ 
is non-zero and it is the desired minimiser 
due to the identity $Rv=v$ (implied by $R^4=I$).
Finally, if $\beta=-1$ (\ie\ $R^2u = -u$), then the function 
$
  w := \psi + i R\psi
$
is non-zero and it is the desired minimiser 
due to the identity $Rw = -iw$.

What is wrong with this argument?
Of course, it remains to show that the sums 
$\psi+R\psi$ and $\psi + i R\psi$ are also minimisers.
This is not obvious because the minimisation problem is not linear.
We are grateful to the anonymous referee
for identifying this gap in our argument
in a previous version of our paper.
\end{Remark}

\subsection*{Acknowledgment}
We are grateful to Dorin Bucur for useful discussions.
D.K.\ was supported
by the EXPRO grant No.~20-17749X
of the Czech Science Foundation.

\textbf{In memory of our colleague and friend Georgi Raikov (1954--2021).}

\subsection*{Data Availability Statements}
The data that supports the findings of this study 
are available within the article.

%\newpage
%\vfill
%--------------%
% BIBLIOGRAPHY %
%--------------%
%
%\addcontentsline{toc}{section}{References}
%\bibliography{bib}
%\bibliographystyle{amsplain}
%

\providecommand{\bysame}{\leavevmode\hbox to3em{\hrulefill}\thinspace}
\providecommand{\MR}{\relax\ifhmode\unskip\space\fi MR }
\providecommand{\MRhref}[2]{%
  \href{http://www.ams.org/mathscinet-getitem?mr=#1}{#2}
}
\providecommand{\href}[2]{#2}

\end{document}